\documentclass[12pt]{amsart}

\usepackage{amssymb,stmaryrd}
\usepackage{booktabs}
\usepackage{cases}
\usepackage{array} 
\usepackage{float} 
\usepackage{rotating}

\usepackage{tikz}

\usetikzlibrary{positioning}
\usepackage{rank-2-roots}

\usepackage{color}
\usepackage{etoolbox} 
\usepackage{expl3}
\usepackage{pgfkeys}
\usepackage{pgfopts}
\usepackage{xparse}
\usepackage{xstring}

\usepackage[foot]{amsaddr}
\makeatletter
\renewcommand{\email}[2][]{%
  \ifx\emails\@empty\relax\else{\g@addto@macro\emails{,\space}}\fi%
  \@ifnotempty{#1}{\g@addto@macro\emails{\textrm{(#1)}\space}}%
  \g@addto@macro\emails{#2}%
}
\makeatother
\usepackage{dynkin-diagrams}
\usepackage{etoolbox}
\def\row#1/#2!{#1_{\IfStrEq{#2}{}{n}{#2}} & \dynkin{#1}{#2}\\}

\usepackage{mathtools}
\usepackage{enumerate,pdflscape}
\usepackage{multirow}
\usepackage{color}
  \usepackage[titletoc]{appendix}
\newcommand{\stt}{\mathfrak{s}_{T, \mathfrak{t}}}

\makeatletter
\newcommand{\svast}{\bBigg@{3}}
\newcommand{\vast}{\bBigg@{4}}
\newcommand{\Vast}{\bBigg@{5}}
\makeatother

\newcommand{\wc}{[ T \cup -T]}
\newcommand{\cw}{[T \cup -T ]}

  \newcommand{\ssimple}{\mathfrak{g}}

  \theoremstyle{definition}
  \newtheorem{definition}{Definition}[section]

   \theoremstyle{plain}

  \newtheorem{lemma}[definition]{Lemma}
  \newtheorem{proposition}[definition]{Proposition}
  \newtheorem{theorem}[definition]{Theorem}
  \newtheorem{corollary}[definition]{Corollary}

    \theoremstyle{definition}

\title[Regular extreme semisimple Lie algebras]{Regular extreme semisimple Lie algebras}
  
\begin{document}

\author[Andrew Douglas]{Andrew Douglas$^{1,2}$}
\address[]{$^1$Department of Mathematics, New York City College of Technology, City University of New York, Brooklyn, NY, USA.}
\address[]{$^2$Ph.D. Programs in Mathematics and Physics, CUNY Graduate Center, City University of New York, New York, NY, USA.}

\author[Joe Repka]{Joe Repka$^3$}
\address{$^3$Department of Mathematics, University of Toronto, Toronto, ON,  Canada.}

\date{\today}


\begin{abstract}
A subalgebra  of a semisimple Lie algebra  is {\it wide} if every simple module of the semisimple Lie algebra remains indecomposable when restricted to the subalgebra. 
A subalgebra  is {\it narrow} if the restrictions of all non-trivial simple modules  to the subalgebra have  proper decompositions.
A semisimple Lie algebra is {\it regular extreme} if any regular 
subalgebra of the semisimple Lie algebra is either narrow or wide. 
Douglas and Repka \cite{dr24} previously showed that  the simple Lie  algebras of type $A_n$ are regular extreme. In this article, we show that, in fact, all simple Lie algebras are 
 regular extreme. 
Finally, we show that no non-simple, semisimple Lie algebra is regular extreme. 
\end{abstract}

\maketitle

\section{Introduction}

A subalgebra  of a semisimple Lie algebra  is {\it wide} if every simple module of the semisimple Lie algebra remains indecomposable when restricted to the subalgebra. 
The term ``wide" was coined by Panyushev  in \cite{panyu}, where he established conditions for large families of subalgebras of semisimple Lie algebras to be wide. 
Douglas and Repka extended this work, in part, by examining conditions that make an arbitrary regular subalgebra of a semisimple Lie algebra wide \cite{dr24}. 

In \cite{dr24}, Douglas and Repka also initiated the study of narrow subalgebras, that is, subalgebras on which the restriction of any non-trivial simple module of the ambient 
 semisimple Lie algebra has a proper decomposition.
In this article, we investigate regular extreme semisimple Lie algebras.  A semisimple Lie algebra is {\it regular extreme} if any regular 
subalgebra of the semisimple Lie algebra is either narrow or wide. 

In \cite{dr24}, Douglas and Repka
 showed that the simple Lie algebras of type $A_n$ are regular extreme using the bases of  $A_n$-type modules 
 created by Feigin, Fourier, and Littelmann \cite{feigin}. This property does not hold for non-regular subalgebras of $A_n$ \cite{dr1}; that is, a non-regular subalgebra
 of $A_n$  may be neither narrow nor wide.  
 
  In the present article, we extend the examination of regular extreme semisimple Lie algebras. We show that all
   simple Lie algebras are regular extreme. Specifically, we establish that the simple Lie algebras of types $B_n$, $C_n$, $D_n$,  $E_6$, $E_7$, $E_8$, $F_4$, and  $G_2$ are regular extreme.
   Then, we prove that  no non-simple, semisimple Lie algebra is regular extreme. In this paper, our proofs do not rely on  specific bases of simple modules, even though this was the approach used 
 with $A_n$ in \cite{dr24}.

The article is organized as follows. Section \ref{background} contains necessary background information. 
In Section \ref{extremesection}, we prove that all simple Lie algebras are regular extreme. 
Then, in Section \ref{nos}, we establish that no non-simple, semisimple Lie algebra is regular extreme.

Note that all Lie algebras and modules in this article are over the complex numbers, and  finite-dimensional.

\section{Background, terminology, and notation}\label{background}

In this section, we review  background on semsimple Lie algebras and their modules; closed subsets of root systems and regular subalgebras; simple Lie algebras and Dynkin diagrams; 
and relevant previous results from \cite{dr24, panyu}.

\subsection{Semisimple Lie algebras}

Let  $\mathfrak{g}$ denote a semisimple Lie algebra, and $\mathfrak{h}$ a fixed Cartan subalgebra of $\mathfrak{g}$. The corresponding root system is denoted  $\Phi$, with Weyl group $\mathcal{W}$. For $\alpha \in \Phi$,  $\mathfrak{g}_\alpha$  denotes  the corresponding root space.
The set of positive roots of $\Phi$ is denoted $\Phi^+$,   and $\Delta =\{ \alpha_1,...,\alpha_n\}\subseteq\Phi^+$ is a fixed base of $\Phi$. 
The {\it rank} of $\mathfrak{g}$ is $n$,  the number of simple roots in $\Delta$.

If $\alpha \in \Phi$, we may fix a nonzero $e_\alpha \in \mathfrak{g}_\alpha$. Then  there 
 is a unique $e_{-\alpha} \in \mathfrak{g}_{-\alpha}$, such that $e_\alpha$, $e_{-\alpha}$, and $h_{\alpha} =[e_\alpha, e_{-\alpha}] \in \mathfrak{h}$
 satisfy $[h_\alpha, e_\alpha]=\alpha(h_\alpha) e_\alpha=2 e_\alpha$, and $[h_\alpha, e_{-\alpha}]=-2 e_{-\alpha}$; thus 
 $e_\alpha$, $e_{-\alpha}$, and $h_\alpha$ span a subalgebra 
isomorphic to $\mathfrak{sl}_2$. 
We define $f_{\alpha} \coloneqq e_{-\alpha}$ for $\alpha \in \Phi^+$.

We may naturally associate $\mathfrak{h}$ with its dual space $\mathfrak{h}^*$ via the Killing form $\kappa$. Specifically,  $\alpha \in \mathfrak{h}^*$ 
corresponds to the unique element $t_\alpha \in \mathfrak{h}$ such that $\alpha(h) = \kappa(t_\alpha, h)$, for all $h \in \mathfrak{h}$. We have the nondegenerate symmetric bilinear form  on $\mathfrak{h}^*$ given by
$(\alpha, \beta) \coloneqq \kappa(t_{\alpha}, t_{\beta} )$, and we may define  $\langle \alpha, \beta \rangle \coloneqq \frac{2(\alpha, \beta)}{(\beta, \beta)} =\alpha (h_{\beta})$, where 
$h_{\beta} \coloneqq \frac{2t_\beta}{\kappa(t_\beta, t_\beta)}=\frac{2t_\beta}{(\beta, \beta)}$.

A root system $\Phi$ is {\it irreducible} if it cannot be partitioned into  two proper, orthogonal subsets. 
If $\Phi$ is irreducible, then $\mathfrak{g}$ is a simple Lie algebra.

The set of weights relative to the root system $\Phi$ is denoted $\Lambda$,  and $\Lambda^+$ is the set of all dominant weights with respect to $\Delta$. 
Let $\lambda_1,..., \lambda_n$ be the {\it fundamental dominant weights} (relative to $\Delta$).
A dominant  weight $\lambda \in \Lambda^+$ may be written as $\lambda = m_1 \lambda_1+ \cdots +m_n \lambda_n$, where $m_i$ is a nonnegative integer for each $i$. For each dominant weight
$\lambda$, $V(\lambda)$ is  the simple $\mathfrak{g}$-module of highest weight $\lambda$. Fix a highest weight vector $v_\lambda \in V(\lambda)$, unique up to scalar multiple.

For an arbitrary $\mathfrak{g}$-module $V$, let $\Pi(V)$ be the set of weights of $V$. Given $\mu \in \Pi(V)$, let $V_\mu =\{ v\in V ~|~ h \cdot v = \mu(h) v, ~ \text{for all}~ h\in \mathfrak{h} \}$. Then, $V$
decomposes into weight spaces
\begin{equation}
 V = \bigoplus_{\mu \in \Pi(V) } V_\mu.
\end{equation}

\subsection{Closed subsets of root systems and regular subalgebras}

A subset $T$ of the root system $\Phi$ is {\it closed} if for any
$x, y \in T$, $x+y \in \Phi$ implies $x+y \in T$. 
Any closed set $T$ is a disjoint union of its {\it symmetric} component $T^r =\{\alpha \in T | -\alpha \in T  \}$, 
and its {\it special} component $T^u=\{ \alpha \in T | -\alpha  \notin T \}$.  Let $S \subseteq \Phi$, then the {\it closure} of $S$, denoted $[S]$, is the smallest closed subset of $\Phi$ containing $S$. 

Let $T\subseteq \Phi$ be a closed subset. Let $\mathfrak{t}$ be a subspace of $\mathfrak{h}$
containing $[\mathfrak{g}_\alpha, \mathfrak{g}_{-\alpha}]$ for each $\alpha \in T^r$. Then
\begin{equation}\label{reggg}
\mathfrak{s}_{T,\mathfrak{t}} = \mathfrak{t} \oplus \bigoplus_{\alpha \in T} \mathfrak{g}_\alpha  
\end{equation}
is a regular subalgebra of $\mathfrak{g}$.  Moreover, all regular subalgebras of $\mathfrak{g}$ normalized by $\mathfrak{h}$ arise in this manner. 
Further, any regular subalgebra of $\mathfrak{g}$ is conjugate under the adjoint group  of $\mathfrak{g}$ to a regular subalgebra normalized by $\mathfrak{h}$. Hence, 
we'll assume that any regular subalgebra is normalized by $\mathfrak{h}$, and thus in the form of Eq. \eqref{reggg}.

A Lie algebra may be either semisimple, solvable, or Levi decomposable  (Levi's Theorem [\cite{levi}, Chapter II, Section $2$]). 
A regular semisimple subalgebra is given by $\mathfrak{s}_{T, \mathfrak{t}}$ for some   symmetric closed subset $T$ (non-empty)  of $\Phi$, and the subalgebra $\mathfrak{t}$ of $\mathfrak{h}$ generated by $[\mathfrak{g}_\alpha, \mathfrak{g}_{-\alpha}]$ for all $\alpha \in T\cap \Phi^+$.
A regular solvable subalgebra is given by a subalgebra  of the form $\mathfrak{s}_{T,\mathfrak{t}}$, where $T$ is a special closed subset of $\Phi$ (possibly empty), and $\mathfrak{t}$ a subalgebra of $\mathfrak{h}$. 
 A regular Levi decomposable subalgebra is a non-semisimple regular Lie algebra $\mathfrak{s}_{T, \mathfrak{t}}$, 
such that $T$ is a closed subset, $T^r$ is the corresponding (non-empty) symmetric closed subset,  $T^u$ is the corresponding  special closed subset of $\Phi$, and  $\mathfrak{t}$ is  a subalgebra of $\mathfrak{h}$ containing $[\mathfrak{g}_\alpha, \mathfrak{g}_{-\alpha} ]$ for each $\alpha \in T^r$.

Lemmas \ref{wideclosure}, and \ref{semireg}; and Proposition \ref{wprop} below are  relevant results for closed subsets of root systems and regular subalgebras.

\begin{lemma}\label{wideclosure}\cite{dr24}
Let $S$ be a closed subset of $\Phi$. Suppose $\beta_1$, $\beta_2$,...,$\beta_k \in S$, and $\beta_1+\beta_2+ \cdots+\beta_k \in \Phi$. Then $\beta_1+\beta_2+\cdots +\beta_k \in S$. 
\end{lemma}

\begin{lemma}\label{semireg}\cite{dr24}
Let $T$ be a closed subset of $\Phi$. Then, $[T\cup -T]$ is a symmetric closed subset of $\Phi$ containing $T$. 
\end{lemma}

In the following proposition, $G$ is the adjoint group of $\mathfrak{g}$. 
The Weyl group $\mathcal{W}$ of $\Phi$ naturally acts on the dual space $\mathfrak{h}^*$. Further, by identifying  $\mathfrak{h}$ and $\mathfrak{h}^*$ via the Killing form, the Weyl group
also acts on $\mathfrak{h}$. 

\begin{proposition}\label{wprop}[ \cite{dougdeg}, Proposition 5.1]
The regular subalgebras $\mathfrak{s}_{T_1, \mathfrak{t}_1}$ and  $\mathfrak{s}_{T_2, \mathfrak{t}_2}$ are conjugate under $G$ if and only if there is a $w \in \mathcal{W}$ with 
$w(T_1)=w(T_2)$ and $w( \mathfrak{t}_1)=w( \mathfrak{t}_2)$.
\end{proposition}

\subsection{Simple Lie algebras and Dynkin diagrams}

As mentioned previously, simple Lie algebras correspond to irreducible root systems. Semisimple Lie algebras are the direct sums of simple Lie algebras.

Given an irreducible root system $\Phi$ with base $\Delta=\{\alpha_1, \alpha_2,...,\alpha_n\}$, we associate a {\it Dynkin diagram}, which is a graph
having a vertex corresponding to each simple root. Between  vertices $\alpha_i$ and $\alpha_j$, we place 
$\langle\alpha_i, \alpha_j \rangle \langle\alpha_j, \alpha_i \rangle $ edges. It can be shown that between two vertices  there are either one, two, three, or no edges \cite{humphreys}.
Whenever a double or triple edge occurs, an arrow is added pointing to the shorter of the two roots.

The Dynkin diagrams for all simple Lie algebras are provided in Figure \ref{alldynkin}. Namely,  the figure contains the Dynkin diagrams of the simple Lie algebras of types
$A_n$, $B_n$, $C_n$, $D_n$,  $E_6$, $E_7$, $E_8$, $F_4$, and  $G_2$. Note that the subscripts of the Lie algebra types denote the rank of the simple Lie algebra.

\begin{figure}[h!]
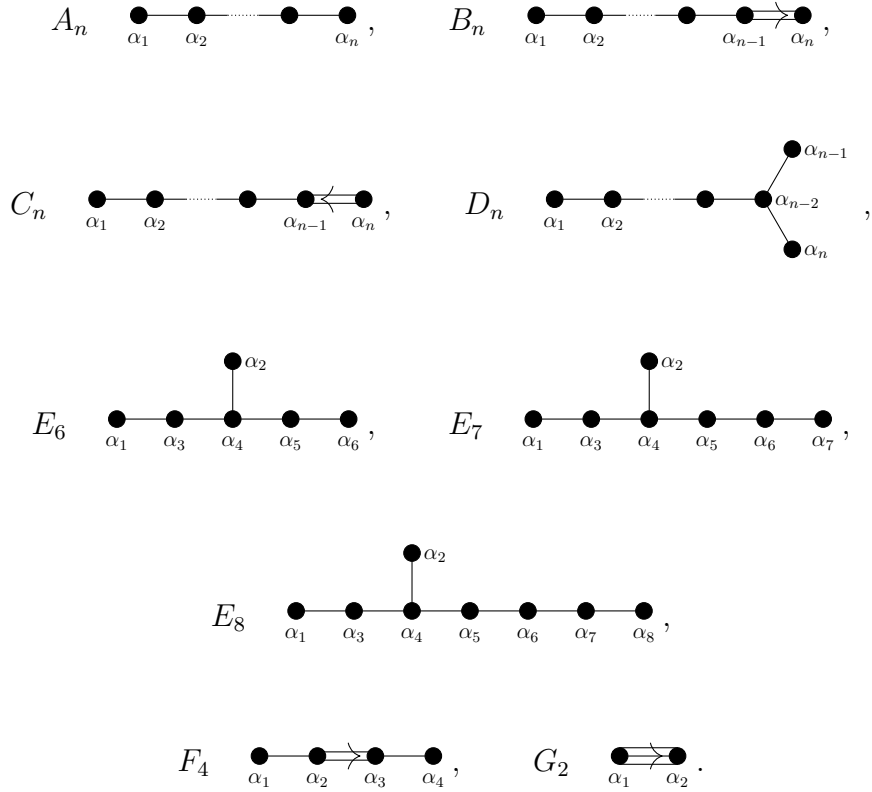

\(A_n  \quad \dynkin[scale=2.2, arrow width = 1.71mm, labels={ \alpha_1,\alpha_2, ,\alpha_{n} }]A{}\), \qquad \(B_n \quad   \dynkin[scale=2.2, arrow width = 1.71mm,  labels={\alpha_1, \alpha_2, , \alpha_{n-1},\alpha_n} ]B{}  \), \\
\vspace{10mm}
\(C_n  \quad \dynkin[scale=2.2, arrow width = 1.71mm, labels={\alpha_1, \alpha_2, , \alpha_{n-1},\alpha_n} ]C{}\), \qquad \(D_n \quad   \dynkin[scale=2.2, arrow width = 1.71mm, labels={\alpha_1, \alpha_2, ,  \alpha_{n-2},\alpha_{n-1}~,\alpha_n}, label directions={, , ,  right,,} ]D{}  \), \\
\vspace{10mm}
\(E_6  \quad \dynkin[scale=2.2, arrow width = 1.71mm, labels={\alpha_1, \alpha_2, \alpha_3,\alpha_4,\alpha_5,\alpha_6} ]E{6}\), \qquad \(E_7 \quad   \dynkin[scale=2.2, arrow width = 1.71mm, labels={\alpha_1, \alpha_2, \alpha_3,\alpha_4,\alpha_5,\alpha_6,\alpha_7}  ]E{7}  \), \\
\vspace{10mm}
\(E_8 \quad   \dynkin[scale=2.2, arrow width = 1.71mm, labels={\alpha_1, \alpha_2, \alpha_3,\alpha_4,\alpha_5,\alpha_6,\alpha_7, \alpha_8}  ]E{8}  \), \\
\vspace{12mm}
\(F_4 \quad   \dynkin[scale=2.2, arrow width = 1.71mm, labels={\alpha_1, \alpha_2, \alpha_3,\alpha_4}  ]F{4}  \), \qquad \(G_2  \quad \dynkin[scale=2.2, arrow width = 1.71mm, labels={\alpha_1, \alpha_2} ]G{2}\).
\caption{Dynkin diagrams of the simple Lie algebras.}\label{alldynkin}
\end{figure}

A {\it simple root path} in a Dynkin diagram of an irreducible root system is a sequence of distinct simple roots $( \beta_1, \beta_2,..., \beta_k)$
such that the vertices associated with the simple roots $\beta_i$ and $\beta_{i+1}$, for all $i$ with $1\leq i \leq k-1$, are connected by at least one edge.

Note that in a simple root path $( \beta_1, \beta_2,..., \beta_k)$, we have $\langle \beta_i, \beta_{i+1}\rangle <0$ for all $i$ with $1\leq i \leq k-1$
since, by definition, the simple roots $\beta_i$ and $\beta_{i+1}$ correspond to vertices in the Dynkin diagram connected by at least one edge, and the fact that 
$\langle \alpha, \gamma \rangle \leq0$ for all simple roots $\alpha$ and  $\gamma$ [\cite{humphreys}, Lemma $10.1$].
Note also that $\langle \beta_i, \beta_{j}\rangle =0$ for all $i$ and $j$ with $| i-j | >1$, since, in this case, the vertices corresponding to 
$\beta_i$ and $\beta_{j}$ are not connected by an edge. We illustrate a simple root path of $E_8$ in Figure \ref{rootpathe8}.

\begin{figure}[h!]
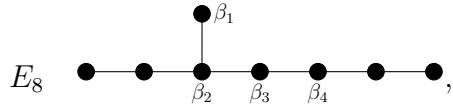

\(E_8 \quad   \dynkin[scale=2.2, arrow width = 1.71mm, labels={, \beta_1, ,\beta_2,\beta_3,\beta_4,, }  ]E{8}  \), \\
\caption{A simple root path $(\beta_1, \beta_2, \beta_3, \beta_4)$ of  $E_8$.}\label{rootpathe8}
\end{figure}

\begin{lemma}\label{catsmeow}
Let  $( \beta_1, \beta_2,..., \beta_k)$ be a simple root path  in a Dynkin diagram for an irreducible root system $\Phi$. 
Then, 
$\beta_i+\beta_{i+1}+\cdots+\beta_j \in \Phi$ for all $i$ and $j$ such that $1\leq i\leq j\leq k$.
\end{lemma}
\begin{proof}
For $i$ and $j$ such that $1\leq i\leq j\leq k$, consider $\beta_i+\beta_{i+1}+\cdots+\beta_j$.  
We have $\langle \beta_p, \beta_q\rangle =0$ for $p$, and $q$ with $|p-q|>1$, as the vertices $\beta_p$ and $\beta_q$ are not connected by an edge.
Since $\beta_i$ and $\beta_{i+1}$ are connected by an edge for each $i$ with $1\leq i \leq k-1$, then $\langle \beta_i, \beta_{i+1}\rangle <0$.  This implies that $\beta_i+\beta_{i+1} \in \Phi$ [\cite{humphreys}, Lemma 9.4]. 

Next, $\langle \beta_i+\beta_{i+1}, \beta_{i+2} \rangle = \langle \beta_i, \beta_{i+2} \rangle+\langle \beta_{i+1}, \beta_{i+2} \rangle <0$. Once again this implies
that $\beta_i+\beta_{i+1}+\beta_{i+2}\in \Phi$. Continuing in this manner, we have $\beta_i+\beta_{i+1}+\cdots+\beta_j \in \Phi$, as required.
\end{proof}

\subsection{Previous results}

Finally, we record relevant results  from  \cite{dr24, panyu} that will be used in the present article. First, however, we need two additional definitions. The first definition is for a 
$\lambda$-wide subalgebra.  Suppose $V(\lambda)$ is the simple $\mathfrak{g}$-module of highest weight $\lambda$. A subalgebra  is  $\mathit{\lambda}${\it -wide} if the simple module  of $\mathfrak{g}$ of highest weight $\lambda$ remains indecomposable when restricted to the subalgebra.

Let $T$ be  a closed subset of $\Phi$. Then, define an $\mathfrak{s}_{\wc, \mathfrak{h}}$-submodule of $V(\lambda)$: 
\begin{equation}
\cw \cdot \lambda \coloneqq \text{Span}\big\{ e_{-\beta_1} \cdots e_{-\beta_m}\cdot v_\lambda ~|~ -\beta_j \in \cw \cap \Phi^-\big\},
\end{equation}
where $j=1,...,m$, and such that $-\beta_1,...,$ and $-\beta_m$  
 are not necessarily distinct.  Note that we include $v_\lambda \in \cw \cdot \lambda$.
In addition, note that if  $e_{-\beta_1} \cdots e_{-\beta_m}\cdot v_\lambda  \neq 0$, then $\lambda -  \beta_1-  \cdots - \beta_m \in \Pi(V(\lambda))$.   %

We are now ready to present the important results from \cite{dr24, panyu} that establish necessary and sufficient conditions for a regular subalgebra to be $\lambda$-wide, and wide, respectively.

\begin{theorem}\label{lwide}\cite{dr24}
Let $T$ be  a closed subset of $\Phi$, $\mathfrak{t}$  a subalgebra of $\mathfrak{h}$ containing $[\mathfrak{g}_\alpha, \mathfrak{g}_{-\alpha} ]$ for each $\alpha \in T^r$, and $\lambda \in \Lambda^+$. Then, $\mathfrak{s}_{T, \mathfrak{t}}$ is $\lambda$-wide if and only if
$\cw \cdot \lambda =  V(\lambda)$.   
\end{theorem}

\begin{corollary}\label{lwideb}\cite{dr24, panyu}
Let $T$ be  a closed subset of $\Phi$, and $\mathfrak{t}$  a subalgebra of $\mathfrak{h}$ containing $[\mathfrak{g}_\alpha, \mathfrak{g}_{-\alpha} ]$ for each $\alpha \in T^r$.  Then, $\mathfrak{s}_{T, \mathfrak{t}}$ is wide if and only if $[T \cup -T] =\Phi$. 
\end{corollary}

Note that necessary and sufficient conditions for essentially all regular solvable subalgebras to be wide was established in \cite{panyu}. The result was extended
to all regular subalgebras in \cite{dr24}.

\section{Regular extreme simple Lie algebras}\label{extremesection}

In this section, we prove that all simple Lie algebras are regular extreme. We begin with a lemma pertaining to simple root paths and
elements of simple modules.

\begin{lemma}\label{stringl}
Let $(\beta_1,...,\beta_k)$ be a simple root path of a Dynkin diagram for an irreducible root system, with associated simple Lie algebra $\ssimple$. Further,
let $V(\lambda)$ be the simple $\ssimple$-module of highest weight $\lambda$, and highest weight vector $v_\lambda$.
If $f_{\beta_1} v_\lambda \neq 0$, then $f_{\beta_1+\cdots +\beta_k} v_\lambda \neq 0$.
\end{lemma}
\begin{proof}
Since $f_{\beta_1} v_\lambda \neq 0$, and necessarily $e_{\beta_1} v_\lambda = 0$, we have that $\langle \lambda, \beta_1 \rangle > 0$ by 
$\mathfrak{sl}_2$-theory. By way of contradiction, suppose that $f_{\beta_1+\cdots+\beta_k} v_\lambda = 0$. Since, again, we necessarily 
have $e_{\beta_1+\cdots+\beta_k} v_\lambda = 0$, then $\langle \lambda, \beta_1+\cdots+\beta_k\rangle =0$. This implies $( \lambda, \beta_1+\cdots+\beta_k ) =0$, and
\begin{equation}
( \lambda, \beta_1 )= -( \lambda, \beta_{2}+\cdots+\beta_k ).
\end{equation}
Note that $\beta_{2}+\cdots+\beta_k \in \Phi$ by Lemma \ref{catsmeow}. And since $\langle \lambda, \beta_1 \rangle > 0$, we have
\begin{equation}\label{ghrtt}
\langle \lambda, \beta_2+\cdots+\beta_k \rangle <0,
\end{equation}
which is impossible since $\lambda$ is the highest weight. Hence, it must be the case that $f_{\beta_1+\cdots+\beta_k} v_\lambda \neq 0$, as required.
\end{proof}

\begin{theorem}
The simple Lie algebras are regular extreme.
\end{theorem} 
\begin{proof}
Let $\stt$ be a regular subalgebra of the simple Lie algebra $\ssimple$. Suppose that $\stt$ is not wide. We must show that $\stt$ is narrow.
That is, given an arbitrary simple  $\ssimple$-module $V(\lambda)$ with $\lambda \neq 0$,  we'll
show that $V(\lambda)$ has a non-trivial $\stt$-decomposition.
Since $\stt$ is not wide, $[T \cup -T] \subsetneq \Phi$ by Corollary \ref{lwideb}. This implies that
$\Delta \setminus [T \cup -T]$ is not empty.  

Let $\lambda =m_1\lambda_1+\cdots + m_l \lambda_l+ \cdots + m_n \lambda_n \neq 0$ ($l$ fixed and $1\leq l \leq n$) be a dominant weight, with $m_l>0$. Then, $f_{\alpha_l} v_\lambda \neq 0$.
We proceed in cases to show that $V(\lambda)$ has a non-trivial decomposition with respect to $\stt$.

\vspace{2mm}

\noindent Case $1$. $\alpha_l \notin \cw$:  Since $f_{\alpha_l} v_\lambda \neq 0$, then
 $\lambda-\alpha_l\in \Pi(V(\lambda))$. We claim that $\lambda-\alpha_l \notin \Pi(\cw \cdot \lambda)$. By way 
of contradiction, suppose that $\lambda-\alpha_l \in \Pi(\cw \cdot \lambda)$. 
Then 
$\lambda-\alpha_l  = \lambda -\gamma_1-\cdots -\gamma_p$ for some  $-\gamma_1$,...,$-\gamma_p \in \cw \cap \Phi^-$. Therefore
$\alpha_l =\gamma_1+\cdots +\gamma_p \in \Phi$, which is thus an element of $\cw$ by Lemma \ref{wideclosure}, a contradiction. 

Hence, it must be the case that $\lambda-\alpha_l \notin \Pi(\cw \cdot \lambda)$. Since $\lambda-\alpha_l \notin \Pi(\cw \cdot \lambda)$ and
$\lambda-\alpha_l \in \Pi(V(\lambda))$, Theorem \ref{lwide} implies that $V(\lambda)$ has a non-trivial $\stt$-decomposition.

\vspace{2mm}

\noindent Case $2$. $\alpha_l \in \cw$: Then, there exists a simple root path $(\beta_1,...,\beta_k)$ such that $\alpha_l=\beta_1$, and
such that $\beta_i \in \cw$ if $ 1\leq i \leq k-1$, and $\beta_k \in \Delta \setminus \cw$. Such a simple root path exists
since $\Delta \setminus [T \cup -T]$ is not empty, and any two vertices in Dynkin diagram of an irreducible root system may be joined by such 
a path. Hence, by Lemma \ref{stringl},
\begin{equation}
\begin{array}{llllllll}
f_{\beta_1+\cdots+\beta_k} v_\lambda \neq 0,
\end{array}
\end{equation}
since $f_{\beta_1}  v_\lambda=f_{\alpha_l} v_\lambda \neq 0$.
This implies that $\lambda-\beta_1 -\cdots-\beta_{k-1}-\beta_k \in \Pi(V(\lambda))$. We claim that $\lambda-\beta_1 -\cdots-\beta_{k-1}-\beta_k \notin \Pi(\cw \cdot \lambda)$. By way 
of contradiction, suppose that $\lambda-\beta_1 -\cdots-\beta_{k-1}-\beta_k  \in \Pi(\cw \cdot \lambda)$. 

Then,  
$\lambda-\beta_1 -\cdots-\beta_{k-1}-\beta_k  = \lambda -\gamma_1-\cdots -\gamma_p$ for some  $-\gamma_1$,...,$-\gamma_p \in \cw\cap \Phi^-$. Hence
$\beta_k =\gamma_1+\cdots +\gamma_p-\beta_{1}-\cdots-\beta_{k-1}  \in \Phi$, which is thus an element of $\cw$ by Lemma \ref{wideclosure} (note that
$\gamma_1,..., \gamma_p, -\beta_{1},..,-\beta_{k-1}  \in \cw$), a contradiction. 

Hence, it must be the case that $\lambda-\beta_1 -\cdots-\beta_{k-1}-\beta_k \notin \Pi(\cw \cdot \lambda)$. Since $\lambda-\beta_1 -\cdots-\beta_{k-1}-\beta_k  \notin \Pi(\cw \cdot \lambda)$ and
$\lambda-\beta_1 -\cdots-\beta_{k-1}-\beta_k  \in \Pi(V(\lambda))$, Theorem \ref{lwide} implies that $V(\lambda)$ has a non-trivial $\stt$-decomposition.

With cases $1$ and $2$, we've established that if $\stt$ is not wide, then $\stt$ is narrow, as required.
\end{proof}

\section{Non-simple, semisimple Lie algebras}\label{nos}

In the final section, we show that no non-simple, semisimple Lie algebra is  regular extreme.

\begin{theorem}
No non-simple, semisimple Lie algebra is regular extreme.
\end{theorem}
\begin{proof}
Let $\mathfrak{g} = \oplus_{i=1}^k \mathfrak{g}^i$ be a non-simple, semisimple Lie algebras, with 
$\mathfrak{g}^1$, $\mathfrak{g}^2$,..., and $\mathfrak{g}^k$ simple Lie algebras, and $k>1$.  To show that $\mathfrak{g}$ is not regular extreme,
it suffices to identify a regular subalgebra of $\mathfrak{g}$  which is neither narrow nor wide.
The regular subalgebra of $\mathfrak{g}$ that we'll  show is neither narrow nor wide is the regular simple subalgebra $\mathfrak{g}^1 \subset \mathfrak{g}$. 

Let $V^1(\lambda)$ be the simple $\mathfrak{g}^1$-module of highest weight $\lambda \neq 0$. Let $V^i(0)$ be the
trivial $\mathfrak{g}^i$-module for $i=1,...,k$. Then
\begin{equation}
V \coloneqq V^1(\lambda) \otimes V^2(0) \otimes \cdots \otimes V^k(0)
\end{equation}
is a non-trivial simple $\mathfrak{g}$-module. Since $V|_{\mathfrak{g}^1} \cong V^1(\lambda)$, then the simple $\mathfrak{g}$-module
$V$ is $\mathfrak{g}^1$-indecomposable. Hence, the regular simple subalgebra $\mathfrak{g}^1$ is not narrow.

We now show that $\mathfrak{g}^1$ is also not wide. Towards this end, let $V^2(\eta)$ be the simple $\mathfrak{g}^2$-module
with highest weight $\eta \neq 0$. Then
\begin{equation}
W \coloneqq V^1(0) \otimes V^2(\eta) \otimes V^3(0) \otimes \cdots \otimes V^k(0)
\end{equation}
is a non-trivial simple $\mathfrak{g}$-module. However, $W$  has a non-trivial $\mathfrak{g}^1$-decomposition. In particular,
let $u^i$ be the single basis element of $V^i(0)$, and $u^2_1,...,u^2_m$ be a basis 
of $V^2(\eta)$, so that $m >1$. Then, the following is a non-trivial $\mathfrak{g}^1$-decomposition of $W$:
\begin{equation}
W|_{\mathfrak{g}^1} = \oplus_{i=1}^m \mathbb{C} \{ u^1 \otimes u^2_i \otimes u^3 \otimes \cdots \otimes u^k \}.
\end{equation}
Hence, $\mathfrak{g}^1$ is also not wide, as required. 
\end{proof}

\end{document}